\documentclass{daj}
\usepackage{amssymb, amsmath, amsthm, amsfonts,xcolor, enumerate,thmtools}
\usepackage{tikz}

\newcommand{\ep}{\varepsilon}

\newcommand{\De}{\Delta}
\newcommand{\cL}{\mathcal{L}}
\newcommand{\cH}{\mathcal{H}}

\newcommand{\HH}{\mathcal{H}}
\newcommand{\cP}{\mathcal{P}}

\newcommand{\cF}{\mathcal{F}}
\newcommand{\cC}{\mathcal{C}}

\newcommand{\ZZ}{\ensuremath{\mathbb{Z}}}
\newcommand{\RR}{\ensuremath{\mathbb{R}}}

\newtheorem{theorem}{Theorem}[section]
\newtheorem{claim}[theorem]{Claim}

\newtheorem{lemma}[theorem]{Lemma}
\theoremstyle{definition}

\dajAUTHORdetails{%
  title = {On the number of points in general position in the plane}, 
  author = {J\'ozsef Balogh  and J\'ozsef Solymosi},
  plaintextauthor = {J\'ozsef Balogh  and J\'ozsef Solymosi},
  plaintexttitle = {On the number of points in general position in the plane}, 
  runningtitle = {On the number of points in general position in the plane}, 
  runningauthor = {J\'ozsef Balogh  and J\'ozsef Solymosi},
  copyrightauthor = {J\'ozsef Balogh  and J\'ozsef Solymosi},
  keywords = {General point sets in the plane, $\ep$-nets, Hypergraph container method},
}   

\dajEDITORdetails{%
   year={2018},
   number={16},
   received={1 September 2016},   
   revised={17 April 2017},    
   published={12 October 2018},  
   doi={10.19086/da.4438},       
}   

\begin{document}

\begin{frontmatter}[classification=text]

\title{On the number of points in general position in the plane} 

\author[balogh]{
J\'ozsef Balogh\thanks{Partially supported by NSF DMS-1500121, NSA Grant H98230-15-1-0002,  Arnold O. Beckman Research Award (UIUC Campus Research Board 15006) and TT 12 MX-1-2013-0006.}}
\author[solymosi]{J\'ozsef Solymosi\thanks{
Partially supported by NSERC,  ERC Advanced Research Grant
AdG. 321104 and by Hungarian National Research Grant NK
104183.}}

\begin{abstract}
In this paper we  study some Erd\H{o}s type problems in discrete geometry.
Our main result is that  we show that there is a planar point set of $n$ points such that no four are collinear but no matter how we choose a subset of size $n^{5/6+o(1)} $ it contains a collinear triple. 
Another application studies $\ep$-nets in a point-line system in the plane.

 We prove the existence of some geometric constructions with a new tool, the so-called Hypergraph Container Method. 
\end{abstract}
\end{frontmatter}


\section{Introduction}\label{sec-intro}

We say that a set of $n\geq d+1$ points on the Euclidean space $\mathbb{R}^d$ is in {\it general  position} if no hyperplane contains $d+1$ points. One might expect that if in a set no hyperplane contains $d+2$ points then a large subset of the $n$ points can be selected such that the points of this subset are in general position. 
Determining how many points one can choose in general position in the worst case is a problem that, like many other similar problems, was initiated by Paul Erd\H{o}s \cite{problem3,problem2}. We consider the planar ($d=2$) case first. The best known upper and lower bounds are far from each other and even a sublinear upper bound was hard to achieve. Our  primary goal is  to prove the existence of a  better construction for the upper bound. Furthermore,  we apply similar methods  to other problems as well. 

An important concept in computational geometry and approximation theory in computer science is that of an $\ep$-net. On a base set $A$ of $n$ elements and a family of subsets ${\cal{F}}\subset 2^A$ a set $E\subset A$ is an $\ep$-{\it net} if $E\cap B\neq \emptyset$ for every $B\in \cal{F}$ such that $|B|\geq \ep n$. One generally expects that if $\cal{F}$ has  low complexity then it has a small  $\ep$-net. In a seminal paper, Haussler and Welzl~\cite{emo}  proved that for any range space of VC dimension $d,$ there is an $\ep$-net of size $O((d/\ep)\log{(d/\ep)})$. For the definitions and for further details we refer to \cite{komlos},  Chapter 15 in \cite{PaAg}, Chapter 10 in \cite{Mat}, or the survey paper of related problems \cite{FuPa}. Here we are concerned with a special case.
Answering a question by Matou\v{s}ek, Seidel, and Welzl~\cite{msw},  Alon~\cite{Alon} proved that the minimum possible size of an $\ep$-net for point objects and line ranges in the plane is super-linear in $1/\ep$. Alon used the density Hales-Jewett  theorem for the construction, so his bound was just enough to break linearity.  As he writes at the end of the paper "{\em The problem of deciding whether or not there are natural geometric range spaces of VC-dimension $d$ in which the minimum possible size of an $\ep$-net is $\Omega((d/\ep) \log(1/\ep)$ remains open. It seems plausible to conjecture that there are such examples, and even to speculate that this is the case for the range space of lines in the plane, for appropriately defined planar sets of points."}
We cannot quite answer this question, but we are able to give a better bound, $(1/\ep)\log^{1/3-o(1)}{(1/\ep)}$, for the range space of lines.
In related work, Pach and Tardos~\cite{PachTardos} constructed range spaces induced by axis-parallel rectangles in the plane, in which the size of the smallest  $ \varepsilon $-net is  $ \Omega \left (\frac {1}{\varepsilon }\log \log \frac {1}{\varepsilon }\right )$. By a theorem of Aronov, Ezra, and Sharir \cite{AES}, this bound is tight.

\medskip

Here, we consider problems concerning points in planes. In a follow-up paper we shall consider higher-dimensional variants. Variants of all of our methods  work there as well, but there are more technical details, so we decided to handle them separately.

We shall use the method of hypergraph containers for the proof. This useful method was recently introduced independently by Balogh, Morris and Samotij~\cite{BMS}, and by Saxton and Thomason~\cite{ST}.  Roughly speaking, it says that if a hypergraph $\cH$ has a uniform edge distribution, then one can find a relatively small collection of sets, {\em containers}, covering all independent sets in $\cH$.  One can also require that the container sets span only few edges.  In our applications this latter condition guarantees that all container sets are small. The right geometric construction determines a hypergraph where all large subsets contain an edge (i.e., a collinear triple). 

The container method has been applied in a similar way in the graph setting~\cite{NS} and \cite{RRS},
and in additive combinatorics~\cite{BLS}, and in several  more recent papers: here we list just those that motivated our work. The main contribution of this paper is to demonstrate that this machinery can be used in discrete geometry. It is possible that better results could be obtained, still using this method, if somebody managed to find more useful hypergraphs, or to prove a variant of the container theorem tailored to these particular hypergraphs.\footnote{Motivated by our work, Balogh and Samotij recently proved a variant of the container theorem, which improves the exponent for the $\ep$-net application from $1/3$ to $1/2$.}

We refer the readers to~\cite{BLS, BLST, BMS, ST} for more details and applications on the container method.

{\bf Notation.} The $O(f(n)), w(g(n)), o(h(n))$ notation are standard: $k(n)=O(f(n))$ if there is a constant  $C>0$ that  $k(n) \le C\cdot f(n)$;   $k(n)=\omega(g(n))$ if there is a constant  $C>0$ that $k(n) \ge C\cdot g(n)$; 
 $k(n)=o(h(n))$ if for every constant  $C>0$ we have $k(n) \le C\cdot h(n)$ for every $n$ sufficiently large.
We use logarithms to base $2$ unless otherwise indicated. Throughout the paper we omit floors and ceilings whenever they are not crucial.

\section{Results}

\subsection{Points in general position}

For an $n$-element set, $S\subset {\mathbb{R}}^2,$ let $S'\subset S$ be the largest subset of points in general position and let $\alpha(S)=|S'|.$
We define
$$\alpha(n):= \min \{\alpha(S):\  |S|=n \text{ and } S \text{ contains no four points in a line}\}.$$
Erd\H os~\cite{problem2} proposed the problem to determine  $\alpha(n)$. F\"uredi~\cite{furedi} proved that
\begin{equation}\label{furedibound}
\omega(\sqrt{n\log n})\le \alpha (n)=o(n).
 \end{equation}

 The lower bound was obtained using a general result of Phelps and R\"odl~\cite{PR}, who gave a lower bound on the independence number of partial Steiner Systems.  F\"uredi~\cite{furedi} proved  the upper bound using the Density Hales-Jewett Theorem (DHJ) of Katznelson and Furstenberg~\cite{FuKa1, FuKa2}. In this application F\"uredi takes a random (generic) projection of the combinatorial cube of the combinatorial space $\{1,2,3\}^N$ to the plane. This point set has no collinear four-tuple,  however for any $\delta>0$ if $N$ is large enough then any subset of points of size at least $\delta\cdot 3^N$ contains a collinear triple, a projection of a combinatorial line. The best known bound for DHJ follows from a recent proof  by D.H.J. Polymath \cite{Poly1}. A subset of $\{1,2,3\}^N$ of density $\delta$ contains a combinatorial line if $N$ is at least a tower of 2's of height $O(1/\delta^3).$  
 We improve the upper bound in \eqref{furedibound}  showing the existence of a  point set containing no four points in a line such that every subset of size $n^{5/6+o(1)} $ contains three points on a line.

\begin{theorem}\label{3and4}
As $n$ goes to infinity,
$$ \alpha (n)\ \le \ n^{5/6+o(1)}.$$
\end{theorem}

It is far from clear whether $5/6$ is the right constant in the exponent. We believe it is likely that the point set that we constructed is close to optimal, but we are much less confident that 5/6 is the correct exponent for this set. 

\subsection{$\ep$-nets and weak $\ep$-nets}

Here, for a given small $\ep^*>0$, we shall construct a point set $S$ in the plane such that the smallest $R\subset S$ that covers every line containing at least $\ep^* |S|$ points from $|S|$ has size at least $(1/\ep^*) \log^{1/3-o(1)}(1/\ep^*)$.

\begin{theorem}\label{general}
Let $\ep^*$ be an arbitrary small constant. Then there exists a point system $S$ in the plane, that if $T\subset S$ intersects each line that contains at least $\ep^*|S|$ points, i.e., $T$ is an $\ep^*$-net, then 
$$|T|\ge  \frac {1}{2\ep^*} \log^{1/3}\left(\frac {1}{\ep^*}\right)\log\log \left(\frac {1}{\ep^*}\right)^{-1}.$$
\end{theorem}

Note that the $O(\frac{1}{\ep} (\log \frac{1}{\ep}))$ upper bound holds for any abstract range space with fixed VC-dimension, see \cite{emo}.  

There is a variant of the $\ep$-net problem where the hitting set $T$ does not need to  be a subset of the base set. Given a point set $S$ in the plane, an $T$ subset of the plane is a {\it weak} $\ep$-{\it net} if every line containing at least $\ep|S|$ points of $S$ contains a point of $T$. Note that $T$ is not necessarily a subset of $S$.   We are interested proving the existence of a set $S$ such that the smallest point set that contains a point from every {\bf line} that is incident to at least $\ep|S|$ points of $S$ is large. Alon~\cite{Alon} gave a lower bound for the size of weak nets as well, as in the previous case using the density Hales-Jewett theorem. We also give a construction to bound the size of weak nets, with the additional feature that it works in the projective plane as well.

Here we can prove only a weaker bound than the one for the non-weak case, but it has a similar form.

\begin{theorem}\label{weaknet}
Let $\ep^0$ be an arbitrary small constant. Then there is a $0< \ep^* < \ep^0$ such that the following holds. There exists a point system $S$ in the plane, that if $T$ intersects each line containing at least $\ep^*|S|$ points, i.e., $T$ is a weak $\ep^*$-net, then 
$$|T|\ge  \frac {1}{10\ep^*}  \left( \log\log \frac {1}{\ep^*}\right)^{1/10}.$$
\end{theorem}

\subsection{Decomposing coverage of the plane}

   The point-line duality of the plane maps points into lines and lines into points in such a way that the incidences (and non-incidences) are preserved. 
Let $\cF$ be a family of sets  over $X$. We say that $\cF$ is an $r$-{\it fold covering} of $X$ if for every $x\in X$ there are at least $r$ sets in $\cF$ containing $x$.
A family of sets $\cF$ over $X$ is {\it cover-decomposable} if there exists an $r$ such that any $r$-fold covering of $X$ with members of $\cF$ can be decomposed into two coverings. 
Pach, Tardos and T\'oth~\cite{pachtt}, using the Hales-Jewett Theorem, proved that when $\cF$ is the set of lines, and $X=\RR^2$, then this family is not cover-decomposable.

 In particular, for every $r$ they proved the existence of a finite $X\subset \RR^2$ and a finite  family of lines $\cF$ such that every point of $X$ was covered by at least $r$ lines of $\cF$, but $\cF$ was not 
  decomposable into two coverings. 
    Following their proof, one can read out a density version, i.e., an estimate on how large the   smallest covering subset of $\cF$ must be.

  In Section~\ref{epsnetsection}  for every $c>0$ and $r$ we prove the existence of a point set $S$ in the plane with the property that any subset of $S$ of size $c\cdot |S|$ contains a collinear $r$-tuple. Denote by $\cL$ the collection of lines containing $r$ points from $S$.
The dual  of $S$, denoted $S^*$, is a family of lines, the dual  of $\cL$, denoted $\cL^*$, is a family of points, with the property that each point of $\cL^*$ is covered by exactly $r$ lines of $S^*$. Additionally, for every subset of $S^*$ of size at least $c|S^*|$ there is a point which is covered by $r$ lines from $S^*$, i.e., it is uncovered by the complement of that subset.
 
   In other words, we give an alternative proof of the result of \cite{pachtt}, with a better quantitive bound on $c$ as a function of the size of $S$.

Mani-Lavistka and Pach~\cite{mani} observed that the Lov\'asz Local Lemma could be used to prove the following result.

\begin{theorem} \label{pach}
Let $T$ and $r$ be positive integers satisfying the inequality $e(T+1)< 2^{r-1}$.
If $\cP$ is a finite point system and $\cL$ is a finite line system in the plane such that every point is covered at least  $r$ times and every line is intersected (in a point from $\cP$) by at most $T$ lines, then $\cL$ can be decomposed into two covering systems.
\end{theorem}

A slight modification of the construction from Section~\ref{epsnetsection}  shows some limits on the bounds in Theorem~\ref{pach}, but it is not clear what the proper order of magnitude of $T$ should be.

\begin{theorem} \label{const}
Let $r$ be a positive integer, and $T:=r^{5r^2}$.
Then there is  finite point system $\cP$ and a finite line system $\cL$ in the plane  that every point is covered at least  $r$ times, and every line is intersected (in a point from $\cP$) by at most $T$ lines from $\cL$, and  $\cL$ cannot be decomposed into two covering systems.
\end{theorem}

For completeness we sketch the proof of Theorem~\ref{pach}. Form a hypergraph with vertex set $\cL$ and edge set $\{F_x:\ x \in \cP\}$, where $F_x$ is the collection of lines containing $x$. By the  Lov\'asz Local Lemma it is $2$-colorable if no $F_x$ is intersected by more than $T$ other hyperedges.\\

We  prove the dual statement of Theorem~\ref{const}. Note that here when we say that a collection of points $\cP$ cannot be decomposed into two covering systems of a family of lines $\cL$, we mean that $\cP$ cannot be partitioned into $\cP_1$ and 
$\cP_2$ in such a way that every line of $\cL$ contains a point from both $\cP_1$ and 
$\cP_2$.

\begin{theorem} \label{dconst}
Let $r$ be a positive integer, and let $T=r^{5r^2}$.
Then there is a finite line system $\cL$ and a finite point system $\cP$ in the plane such that \\
(i) every line contains  at least  $r$ points, \\
(ii)  every point from $\cP$ shares a line   from $\cL$ by at most $T$ points from $\cP$, \\
(iii)  $\cP$ cannot be decomposed into two covering systems.
\end{theorem}

\subsection{Chromatic number of cell-arrangements in the plane}

A simple arrangement  of a set $\cL$ of lines in $\RR^2$ is when there are 
 no parallel lines and no three lines
going through the same point. It decomposes the plane into a set $\cC$ of cells, i.e.~maximal
connected components of $\RR^2\setminus\bigcup_{\ell\in \cL} \ell$. In~\cite{hurtado}, the  hypergraph $\HH_{\text{line-cell}} = (\cL;\cC)$ was defined,
with the vertex set $\cL$, and with each hyperedge $F\in \cC$ given
by a set of lines forming the boundary of a cell of $\cL$. It was proved in~\cite{hurtado} that there is a family $\cL$ of $n$ lines such that the corresponding hypergraph $\HH_{\text{line-cell}}$ has chromatic number  $\Omega(\log n/\log\log n)$, and in~\cite{ackerman} that it is always $O(\sqrt{n/\log n})$.
 In~\cite{ackerman}, the following connection to the Erd\H os $(3,4)$-problem was pointed out. Suppose that we have $n$ points in a plane, with no four in a line, such that any $t$ of them contain three that are on a line. Then one can take the dual of the point system to obtain a family of $n$ lines such that no four of them intersect in a point, but any $t$ of them contain three that have a common point. A small perturbation of the lines changes every intersecting triplet of lines into a triangle. Thus, having a good construction for the  Erd\H os $(3,4)$-problem automatically gives a density version  of~\cite{hurtado}. To summarize, Theorem~\ref{3and4} instantly implies the following result.

\begin{theorem}\label{chrom}
There is a set of $n$ lines in the plane such  that any subset of it of size $ n^{5/6+o(1)} $ completely  contains a cell.
In particular, the cell-chromatic number of the system is at least $n^{1/6-o(1)}$.
\end{theorem}

 \section{The hypergraph container theorem}\label{subsec-container}

Now let us present our main tool, the hypergraph container theorem.
Let $\cH$ be an $r$-uniform hypergraph with average degree $d$. For every $S\subseteq V(\cH)$, its co-degree, denoted by $d(S)$, is the number of edges in $\cH$ containing $S$, i.e.,~
$$d(S)=|\{e\in E(\cH): S\subseteq e\}|.$$ 
For every $j\in[r]$, denote by $\De_j$ the $j$-th maximum co-degree, i.e.,~
$$\De_j=\max\{d(S): S\subseteq V(\cH), |S|=j\}.$$
For  $\tau\in (0,1)$, define 
$$\De(\cH,\tau)=2^{\binom{r}{2}-1} \sum_{j=2}^r \frac{\De_j}{d\tau^{j-1} 2^{\binom{j-1}{2}}} .$$
In particular,  when $r=3$
$$
\De(\cH,\tau)=\frac{4\De_2}{d\tau} + 
\frac{2\De_3}{d\tau^{2}}. $$ 

We are going to use the following version of the hypergraph container theorem (Corollary~3.6 in~\cite{ST}).

\begin{theorem}\label{thm-container}
Let $\cH$ be an $r$-uniform hypergraph with vertex set $[N]$. Let $0<\ep,\tau<1/2$. Suppose that $\tau<1/(200\cdot r \cdot r!^2)$ and $\De(\cH,\tau)\le \ep/(12r!)$. Then there exists $c=c(r)\le 1000\cdot r\cdot r!^3$ and a collection $\cC$ of vertex subsets such that

(i) every independent set in $\cH$ is a subset of some $A\in \cC$;

(ii) for every $A\in\cC$, $e(\cH[A])\le\ep \cdot e(\cH)$;

(iii) $\log|\cC|\le cN\tau\cdot \log(1/\ep)\cdot \log(1/\tau)$.
\end{theorem}


\section{Supersaturation of collinear point sets in the grid.}

Given  integers $n,k,r\ge 3$, we consider the following $r$-uniform hypergraph $\cH=\cH(n,k,r)$ encoding the set of all collinear $r$-tuples in $[n]^k=\{1,2,\ldots,n\}^k$. We let  $V(\cH)=[n]^k$ and we take the edge set  of $\cH$ to consist of all collinear $r$-tuples. 

While it is easy to compute the number of vertices of $\cH$, $|V(\cH)|=n^k=:N$, it takes more effort to estimate the number of edges. For the bounds below, we shall assume that $r,k\le 0.01\cdot \log n$.

\begin{claim}\label{edgecounting}
(i) If $r\le k$ then 
\begin{equation}
\frac{n^{2k}}{r^{2k}}\le \ e(\cH) \le \frac{k\cdot  2^{r+k} }{r!} n^{2k}.
\end{equation}
(ii) If $r =  k+1$ then 
\begin{equation}
\frac{n^{2k}}{r^{2k}} \le \ e(\cH) \le \frac{k \cdot 2^{r+k} }{r!} \cdot n^{2k}\log n.
\end{equation}
(iii) If $ k+1< r\le 2k$ then 
\begin{equation}
  k\cdot \binom{n}{r} \cdot n^{k-1}  \le  \ e(\cH) \le \frac{k\cdot  2^{r+k+1} }{r!} \cdot n^{r+k-1}.
\end{equation}
\end{claim}

\begin{proof} For the lower bounds, we show the existence of many collinear $r$-tuples.
A collinear $r$-tuple can be written in a form $u,u+v,\ldots, u+(r-1)v$ for some $u,v\in \ZZ^k$. In order  to keep the points in $[n]^d$,   we shall choose $u,v\in [n/r]^k$, which proves (i) and (ii). For (iii),  
when $r>k$, we can do better, as    in any of the $k$ (axis-parallel)  directions we have $n^{k-1}$ parallel lines, each containing $n$ points, i.e.~$\binom{n}{r}$ collinear $r$-tuples.

For the upper bounds we need a bit more care. Write $L(t)$ for the number of lines containing more than  $2^t$ and at most $2^{t+1}$  points.
Then we want to bound

\begin{equation}\label{pointlower}
\sum_{t=1}^{\log n} L(t)\cdot \binom{2^{t+1}}{r}.
\end{equation}

If a line contains exactly $\ell$ points, where $2^t <  \ell \le 2^{t+1}$,  then  we can list these $\ell$ points  as  $u,u+v,\ldots, u+(\ell-1)v$ for some $u,v\in \ZZ^k$. Observe that because these points are in $[n]^k$, the absolute value of each coordinate of $v$ is at most $n/(\ell-1)$. By symmetry, and for the price of a factor of $2^k$, we can assume that each coordinate of $v$ is non-negative. This implies that at least one coordinate of $u$ is at most $n/(\ell-1)$, otherwise $u-v$ was on the line as well.
Using that   $2^t\le \ell < 2^{t+1}$ this means that   at least one coordinate of $u$ is at most $n/2^t$,
 and  each coordinate of $v$ is at most $n/2^t$, meaning 
    that the number of choices for $v$ is at most $(n/2^{t})^k$, and for $u$ is at most 
$k \cdot n^k/2^t$, therefore $L(t)\le 2^k\cdot (n/2^{t})^k \cdot k \cdot n^k/2^t$.
Writing this into \eqref{pointlower} we obtain the following upper bound on the number of collinear $r$-tuples:

\begin{equation}\label{pointlowe2}
\sum_{t=1}^{\log n} 2^k\cdot (n/2^{t})^k \cdot k \cdot n^k/2^t \cdot \binom{2^{t+1}}{r}\le 
\sum_{t=1}^{\log n} \frac{k \cdot 2^{r+k} }{r!} n^{2k}2^{t(r-k-1)}.
\end{equation}


When $r\le k$ then
$\sum_t 2^{t(r-k-1)}\le 1.$ 
  When $r=k+1$, then $\sum_t 2^{t(r-k-1)}\le \log n.$  If $r>k+1$    then $\sum_t 2^{t(r-k-1)}\le 
2\cdot n^{r-k-1}$.
\end{proof}

A key ingredient is the supersaturation lemma. First we consider the $k=r=3$ case.

\begin{lemma}\label{supsatkr3}
 In $\cH(n,3,3)$ every set of vertices of size $n^{3-s}$ spans at least $ \frac{ n^{6-4s}}{ 3\cdot 10^{9}\log n}$ hyperedges, where $0\le s<1 $ and $n$ is sufficiently large.
\end{lemma}

\begin{proof} Let  $ S$ be an arbitrary subset of vertices of size $|S|=n^{3-s},$
where  $0\le s<1$. Let $ t = 2000n^s$ and 
$$U:=\{(a,b,c): \ 1\le a\le n/t,\quad  -n\le b,c\le n\}, $$
$$\quad V=\{(a,b,c):\ a\ge |b|, |c|,\quad n/t\le a\le 2n/t, \quad a \text{ is a prime}\}.$$

We define the collection of  lines $\cL=\cL(t)$, whose starting points are in $U$, and their directions are in $V$ (note that we define these lines in $\RR^3$, and we are interested in their intersection with $[n]^3$). 
The line system has several properties.

\begin{claim}\label{inci3}
(i)   Every line in $\cL$ contains at most $t$  points from $[n]^3$.\\
(ii)  $|\cL |\le 300n^6/(t^4\cdot \log (n/t)).$\\
(iii) Every point in $[n]^3$   is  contained  in at least  $n^3/( t^3\log (n/t))$ and at most $50n^3/(t^3\cdot \log (n/t))$  lines from $\cL$. 
\end{claim}

\noindent {\it Proof of claim.} Part (i) follows, as the absolute value of the first coordinate of a direction is at least $n/ t$. Part (ii) follows from considering all choices from  $U$ and $V$, and that the number of prime numbers smaller than $x$ is at least $0.9x/\log x$ and at most $1.1x/\log x$ for $x$ sufficiently large.

For the lower bound in (iii); observe that each $(a,b,c)\in V$ with $b,c\ge 0$, the collection of lines with slope $(a,b,c)$ and starting point from $U$ cover $[n]^3$, hence each point of  $[n]^3$ is covered as many lines as the number of such triplets. Clearly, $|V|$ is an upper bound. This proves the claim.

The number of point-line incidences of the points from $S$ and lines from $\cL$ is, by Claim~\ref{inci3}~(iii), at least $  n^{3-s} \cdot n^3/(  t^3 \log (n/t))$, so 
the average number of points per line is at least  $t\cdot n^{-s}/300>6$. Therefore, the number of collinear triplets in $S$ is at least

$$\binom{t\cdot n^{-s}/300}{3}\cdot\frac{ 300n^6}{t^4\cdot \log (n/t)}
\ge 
 \frac{n^{6-3s}}{1100000 t\log (n/t)}\ge  
 \frac{n^{6-4s}}{3\cdot 10^9\log (n/t)} .$$
\end{proof}


In the next supersaturation result on $\HH= \HH(n,k,r)$, we  have $r=k$,
 and we are interested in finding many edges in large subsets of vertices. The idea of the proof is similar to the one in Claim~\ref{inci3}, but the computation is somewhat more technical.
  Unfortunately, the result is useful only when $k<0.01\log^{1/2} n$, because of the $2^{3k^2}$ term. It
  might be useful to extend the proof to improve the exponent
    $1/2$ to $1-c$ for some small constant $c>0$.

\begin{lemma}\label{supsatkbr}
Let $ 0.1< \gamma \le 1$  and  $r=k<0.01\log^{1/2} n$ and $n$ be sufficiently large. Then  every set of vertices of size $\gamma n^{k}$ spans
 at least $$ \frac{\gamma^{k+1}\cdot n^{2k}}{k! \cdot 2^{3k^2} \cdot  10\cdot 2^{3k}\cdot k \cdot \log n} $$ hyperedges.
\end{lemma}

\begin{proof} Let  $ S$ be an arbitrary subset of vertices of size $|S|=\gamma n^{k}$ and fix  $ t = 10\cdot 2^{2k}\cdot k/\gamma$. Let
$$U:=\{(a_1,\ldots,a_k): \ 1\le a_1\le n/t,\quad  -n\le a_2,\ldots, a_k\le n\}, $$
$$\quad V=\{(a_1,\ldots,a_k):\ |a_1|\ge |a_2|,\ldots, |a_k|,\quad n/t\le a_1 \le 2n/t, \quad a_1 \text{ is a prime}\}.$$

We define a collection of  lines $\cL=\cL(t)$ with starting points are in $U$ and their directions in $V$. 
The line system has several properties.

\begin{claim}\label{incir}
(i)  Every line in $\cL$ contains  at most $t$ points from $[n]^k$.\\
(ii)  $|\cL |\le 2^{3k}n^{2k}/(t^{k+1}\cdot \log (n/t)).$\\
(iii) The number of lines in $\cL$ containing a given point of $[n]^k$ is at least  $ 2^k\cdot n^k/( t^k\log (n/t))$ and at most 
$4^{k+1}n^k/(t^k\cdot \log (n/t))$. 
\end{claim}

\noindent \textit{Proof of claim.} Part (i) follows, as the absolute value of the first coordinate of a direction is at least $n/ t$. Part (ii) follows from considering all choices from  $U$ and $V$, i.e., 
$$|\cL |\le|U|\cdot |V|\le (2n+1)^{k-1}(n/t)(4n/t+1)^{k-1} (3n/t)/\log(n/t)\le \frac{2^{3k}\cdot n^{2k}}{ t^{k+1}\cdot \log(n/t)},$$ by the Prime Number Theorem for $n$ sufficiently large. 
For the lower bound in (iii); observe that each $(a,b,c)\in V$ with $b,c\ge 0$, the collection of lines with slope $(a,b,c)$ and starting point from $U$ cover $[n]^3$, hence each point of  $[n]^3$ is covered as many lines as the number of such triplets. Clearly, $|V|$ is an upper bound. This proves the claim.

The number of point-line incidences of the points from $S$ and lines from $\cL$ is, by Claim~\ref{incir} (iii), at least $  \gamma\cdot2^k\cdot  n^{2k} /(  t^k\log (n/t))$, so 
the average number of points per line is at least  $  \gamma\cdot2^k\cdot  n^{2k} /(  |\cL| \cdot t^k\log (n/t))\ge             2^{-2k}\cdot \gamma \cdot t $. Therefore, using that  $ t > 10\cdot 2^{2k}\cdot k/\gamma$,  the number of collinear $r$-tuples in $S$ is at least

$$  \binom{ \gamma\cdot2^k\cdot  n^{2k} /(  |\cL| \cdot t^k\log (n/t))}{r}  \cdot  |\cL| \ge     \binom{2^{-2k} t\cdot \gamma}{k}\cdot\frac{  2^{3k}n^{2k}}{t^{k+1}\cdot \log (n/t)}
\ge 
 \frac{\gamma^k\cdot n^{2k}}{k! \cdot 2^{3k^2} \cdot  t\cdot \log n}.$$
 \end{proof}

\section{The container method on $[n]^3$: Proof of Theorem~\ref{3and4}}

We apply the container lemma several times to the hypergraph  $\cH=\cH(n,3,3)$ and to its subhypergraphs. First we let $C^0:=[n]^3$, $\cC^0:=\{C^0\}$ and 
write $\cH^0=\cH[C^0]=\cH$. Applying the container lemma to $\cH^0$ will output  a container family $\cC^1$. For each $C^1_j \in \cC^1$ we consider the hypergraph $\cH^1_j:= \cH[ C^1_j]$  and  we apply the container lemma.  The collection of all container sets obtained from all $C^1_j \in \cC^1$ will form 
$\cC^2$.    We iterate this process, i.e.~we apply the container lemma for the hypergraphs spanned by the  containers and we continue it until each of the container sets is sufficiently small. We shall prove that the number of steps is $O(1)$, so that the total number of containers is still small.
Note that one could form a general container theorem for hypergraphs with a strong supersaturation property along these ideas, but it appears to be better to work on the particular hypergraphs separately. Similar multiphase applications of the container method have appeared for example in  \cite{BMT} and \cite{CM}.

    In the final step of the proof of the theorem, as is usual in the container method, we claim that none the container sets  contains too many elements of a random subset of the vertices, so that we can use a union bound.

\medskip

We fix an arbitrary small constant $f>0$, which will show up in the error term in the exponent next to $5/6$.
Also, the number of iterations will be $O(1/f)$.
For the generic iteration step $i$, assume that we work with $\cH^i_j\subset \cH$ spanned by a container set $C^i_j=V(\cH^i_j)$, where $i$ stands for the number of rounds and $j$ is just a running label over various hypergraphs. Define $0\le s=s(i,j)$ such that

$$v(\cH^i_j)=n^{3-s}.  $$ Note, that in the first step we have $s=0$.
If $s\ge 1/3-f$ then $C^i_j$ is sufficiently small, and we put it into the final container family $\cC$, and we shall not do anything with it.

\medskip

Now we apply the container lemma on $\cH^i_j=\cH[C^i_j]$:
By Lemma~\ref{supsatkr3} we have a lower bound on $e(\cH^i_j)$, and hence on the average degree of 
$\cH^i_j$ as well:

$$e(\cH^i_j)\ge  \frac{ n^{6-4s}}{ 3\cdot 10^{9}\log n} \quad \text{and} \quad d(\cH^i_j)\ge  \frac{3 n^{6-4s}}{ 3\cdot 10^{9}\log n\cdot n^{3-s}} =  \frac{ n^{3-3s}}{ 10^{9}\log n}. $$
Observe that the hypergraph $\cH^0$, and hence each subhypergraph of it  satisfies that
$$\Delta_2 < n,\quad \Delta_3= 1.$$
We set $$\tau=n^{s-4/3} \quad \text{and} \quad  \ep=n^{s-1/3+f/2}< n^{-f/2}. $$

To apply  Theorem~\ref{thm-container} we need the following condition to be satisfied (when $n$ is sufficiently large)

$$\De(\cH^i_j,\tau)=\frac{4\De_2}{d\tau} + 
\frac{2\De_3}{d\tau^{2}}\le \frac{4\cdot\ 10^9\cdot \log n}{n^{2/3-2s}}+ 
\frac{2\cdot\ 10^9\cdot \log n}{n^{1/3-s}}\le \frac{\ep}{72}. $$
By Theorem~\ref{thm-container}, there is a family $\cC^{i+1}_j$ of containers such that
$$|\cC^{i+1}_j|\le 2^{10^6\cdot n^{3-s} \tau\log^2 n} =  2^{10^6\cdot n^{5/3}\log^2 n},$$
and each $C\in \cC^{i+1}_j$ spans at most $\ep\cdot e(\cH^i_j)$ edges.

In the next round, for each $C\in \cC^{i+1}_j$ we run the same process (say $C$ is 
playing the role of $V(\cH^{i+1}_{j'})$). The key point of the iteration is that at  each iteration step, when $i$ is increasing to  $i+1$, then either $e(\cH^i_j)\le n^{14/3+f}$ and  $v(\cH^i_j)\le n^{8/3+f}$ or the number of edges is shrinking by a factor of at least  
$ n^{-f/2}$. Therefore the number of iteration steps is at most $12/f$, so the total number of containers will be at most 
$$ 2^{(12/f)10^7\cdot n^{5/3}\log^2 n}.$$
Here, the key point was that  when we are apply the container lemma to a hypergraph spanned by a container, the number of new containers increases just by a constant factor in the exponent, which accumulates only to a constant factor in the exponent, since the number of iterations is a constant.
The iteration stops, since we do not touch a container $C$ when $|C|\le n^{8/3+f}.$\\

To summarize: we have obtained a family of containers $\cC$ containing at most $ 2^{(12/f)10^7\cdot n^{5/3}\log^2 n}$
sets, each of size at most $n^{8/3+f}.$

\bigskip

Now we turn to the probabilistic construction of our point set in the plane.
 We take a $p$-random   subset of $[n]^3$, keeping each point with probability $p$, independently of the other points. Having this random point set, we remove a point from each collinear $4$-tuple. 
We want to choose a maximum value of $p$ such that most of the randomly chosen points should stay.
The expected number of random points is $pn^3$. By Claim~\ref{edgecounting}~(ii) the expected number of collinear $4$-tuples is 
 at most $100\cdot p^4\log n\cdot n^6$, hence for 
 $$p=\frac{1}{n\log^{1/2} n}$$ we have 
 $$pn^3\gg 100  p^4\log n\cdot n^6.$$
 Denote the resulting set by $S$. 
 
  By the container method above we proved that  there are at most $ 2^{(12/f)10^7\cdot n^{5/3}\log^2 n}$     containers (subsets of the vertex set), each set having size at most $n^{8/3+f}$, and each  independent set being contained in one of them. 
  
  We claim, using the first moment method, that $S$ is unlikely to contain an independent set of size $m$ when $m=n^{5/3+f}$:

 $$2^{(12/f)10^7\cdot n^{5/3}\log^2 n}  \cdot \binom{n^{8/3+f}}{m} p^m\le 2^{(12/f)10^7\cdot n^{5/3}\log^2 n} \cdot
 \left( \frac{e}{\log^{1/2}n}  \right)^m = o(1).$$
 
 Therefore, there is a set $S$ which does not contain an independent set of size $n^{5/3+f}$, does not contain $4$ points in a line, and has the property that
 $$\frac{ n^2}{2\log^{1/2}n} \le |S|\le \frac{2 n^2}{\log^{1/2}n}.$$

 The set $S$ can be projected into the plane in such a way that collinear point tuples stay on a line and no new collinear point tuples are created. Thus we have about $n^2/\log^{1/2}n$ points in the plane such that every subset of size $n^{5/3+f}$  contains three points on a line. As $f>0$ was an arbitrary small constant, this completes the proof of Theorem~\ref{3and4}.

\section{The container method on $[n]^r$: application to $\ep$-nets}\label{epsnetsection}

In the beginning of the proof we apply the   container method as in the previous section,
but to a different hypergraph.
  We complete the proof   using  ideas of Alon~\cite{Alon}. Note that here the application of the container method is somewhat simpler (though the supersaturation part is more technical), because  it is sufficient to obtain containers of size $0.1\cdot n^r$, i.e., a $1/10$th proportion of the vertex set of the hypergraph,     therefore it is sufficient   to apply the container lemma only once.

   Here,  we shall work with $\cH=\cH(n,r,r)$, where $r\le 0.01\log^{1/2} n/\log\log n$, but $r=r(n)\gg \log\log n,$ we assume that both $r$ and $n$ are sufficiently large for the computations to hold in this section: in particular, Lemma~\ref{supsatkbr} holds.
Using the fact that there are at most $n$ points in a line, we have 
$$
 \De_2= \binom{n-2}{r-2},    \quad \ldots ,  \quad \De_i= \binom{n-i}{r-i},  \quad\ldots, \quad  \De_r=1.
$$
Using Claim~\ref{edgecounting} (i) we can estimate the average degree of $\cH$:

$$\frac{n^{2r}}{r^{2r}} \le e(\cH)\le \frac{ 2^{2r}}{(r-1)!}n^{2r} \quad \text{and} \quad 
\frac{n^{r}}{r^{2r-1}} \le d=d(\cH)\le \frac{ r^2 \cdot 2^{2r} }{r!}n^{r}.$$
Also, by Claim~\ref{edgecounting} (ii), the number of collinear $(r+1)$-tuples is at most $\frac{(r+1)\cdot 2^{2r+1}}{r!}n^{2r}\cdot \log n.$

We need the following supersaturation statement, which is implied by Lemma~\ref{supsatkbr}.

 \begin{lemma}\label{sat}
 Let $T\subset V(\cH)$ be a set such that $|T| \ge 0.1 v(\cH)$. Then $$e(\cH[T])\ge  \frac{n^{2r}}{10^{r+2}2^{3r^2+3r}(r+1)!\cdot \log n} \ge 10^{-r^2}  e(\cH).$$
 \end{lemma}

 Now we take a random subset $S$ of $v(\cH)$, choosing each vertex with probability $p$ independently of the other choices. We want the number of collinear $(r+1)$-tuples to be less than the number of vertices, so that removing one point from each collinear $(r+1)$-tuple results in a set that contains no   collinear $(r+1)$-tuple. Thus, we need

$$pn^r\gg p^{r+1}     \frac{(r+1) \cdot 2^{2r+1}}{r!} n^{2r}\log n,$$
observing that $\log^{1/r} n=1+o(1)$, it is satisfied when
\begin{equation}\label{compp}
p =  \frac{r}{20 n   }.
\end{equation}

We set $\ep=10^{-r^2}$, which by Lemma~\ref{sat} means that if a container spans
 fewer than $\ep\cdot e(\cH)$ edges, then it has at most $n^r/10$ vertices.

For the application of the container lemma,  we need to choose $\tau$ such that
\begin{equation}\label{taucon}
\tau<\frac{1}{200\cdot r \cdot r!^2}
\end{equation}
and
\begin{equation}\label{compdelta} \De(\cH,\tau)=2^{\binom{r}{2}-1} \sum_{j=2}^r \frac{\De_j}{d\tau^{j-1} 2^{\binom{j-1}{2}}} \le 
2^{\binom{r}{2}-1} \sum_{j=2}^r \frac{\binom{n-j}{r-j}\cdot r^{2r-1}}{n^r \tau^{j-1} 2^{\binom{j-1}{2}}} \le  \frac{\ep}{12r!} . \end{equation}

To satisfy the above conditions we  choose

$$ \tau= n^{-1- 0.9/(r-1)}.$$

Let us see why this is a good choice. The inequality \eqref{taucon} is easily satisfied, as $r \le \log^{1/2} n$. The bulk of \eqref{compdelta} is that 
$n^{-r}\tau^{1-j} \ll  10^{-r^2}$, which is the hardest to satisfy when $r=j$. Plugging our choice of $\tau$ simplifies it to $n^{0.1}\gg   10^{r^2}$, which is satisfied as $r\ll \log^{1/2} n.$

The Container Theorem, Theorem~\ref{thm-container}, gives the existence of at most  
$$2^{1000\cdot r\cdot r!^3 v(\cH)\tau\cdot \log(1/\ep)\cdot \log(1/\tau)}\le  2^{ n^{-1- 0.8/r} \cdot v(\cH)}$$
containers, each of size at most
$$ 0.1\cdot v(\cH),$$
such that  each independent set is contained in one of them.
Note that the above inequality is satisfied as $n^{0.1/r}\gg r^r\cdot \log n$ for our choice of $r$. This is  where the upper bound for $r$ comes from, though we used in many other places the fact that $r \le \log^{1/2} n$.

Now we can give an upper bound on the number of independent sets in a $p$-random subset of $\cH$ of size $m=0.45 pv(\cH)$, by simply multiplying  the number of containers with the number of choices of an $m$-set from a container:

$$
2^{ n^{-1- 0.8/r} \cdot v(\cH)}
\cdot  \binom{  0.1  \cdot v(\cH)}{m}\cdot p^m\le  2^{ n^{-1- 0.8/r} \cdot v(\cH)} (0.61)^m<
2^{( n^{-1- 0.8/r} -p/5)\cdot v(\cH)} =o(1).$$

Now we can put together the proof.
Consider a  $p$-random set of $v(\cH)$. By the choice of $p$, the 
 removal of $o(pv(\cH))$ points can make it to have the property that no $r+1$ points are in a line.
Project the set into the plane, keeping the collinearity of the points (and more importantly not creating new collinear point sets),  and denote the resulting set by $S$.
Note that we had that with high probability $S$ does not contain a subset of size $|S|/2$ that 
 contains no $r$ points in a line.
Now we choose carefully the parameters:
Let $\ep^*$ be a small positive constant, and set
 
 $$  r:= \log^{1/3}   \left(\frac {1}{\ep^*}\right)\log\log \left(\frac {1}{\ep^*}\right)^{-1} \quad\text { and }\quad  |S|= \frac{r}{\ep^*}.$$
   
   One can compute  $n$, using \eqref{compp} and  $|S|=p\cdot n^r$, but we do not need explicit formula for $n$. Its existence  is sufficient for us  for every small  $\ep^*>0$. It may be useful to state that
   $$\ep^*=2^{-(1+o(1))r^3}, \quad n= 2^{(1+o(1))r^2},\quad p=  2^{-(1+o(1))r^2}.$$

   If $T\subset S$ is an  $\ep^*$-net, then it has to intersect every line containing $r$ points,
which means that $S-T$ is an independent set. Note that this is the part of the proof where we use the fact that there is no line containing $r+1$ points. There is no independent set of size $|S|/2$ in $S$, so  $|S|-|T|<|S|/2$, which means that 
$$|T|\ge \frac{|S|}{2}\ge \frac {1}{2\ep^*} \log^{1/3}\left(\frac {1}{\ep^*}\right)\log\log \left(\frac {1}{\ep^*}\right)^{-1}.$$
\hfill Q.E.D.

\section{Weak $\ep$-nets}\label{weak}
 We also have a construction to bound the size of weak nets that works in the projective plane as well. The projection of $[k]^n$ has a very small weak $\ep$-net if one can use points in infinity, i.e., if we consider the problem in the projective plane. There are $2^n-1$ different directions for the lines, so $2^n-1$ points can cover all $k$-rich lines. In our construction after the tilting and projection  the rich lines are in general position.

Here we can prove a similar bound to that for the non-weak case, but it is weaker. Our main trouble is that the points are in too many lines, and there are far too many eliminated points to handle.
We use a similar approach, but we consider a much sparser hypergraph. Denote by $\cF=\cF(n,k,r)$ the following hypergraph. Its vertex set is $[k]^n$, and an $r$-tuple of vertices forms a hyperedge if the vertices are on an axis-parallel line: that is, all but one coordinate of their points are the same.

As before, we shall project it to the plane.
We use a map that keeps any three or more points collinear if they were on an axis-parallel line, and does not create any new collinear $3$-tuples in the process. Here we give a technical attempt to describe this map, which can be skipped by the reader. To avoid non axis-parallel triples being collinear we use a generic Cartesian product. Let  $H_i=\{a^i_1, a^i_2, \ldots , a^i_k\}$ where all $kn$ entries, $a^i_j$-s, are algebraically independent, distinct, non-zero real numbers. Now let $V(\cH)$ denote the Cartesian product $H_1\times H_2\times \ldots \times H_n.$ In this ``grid'' only axis-parallel triples are collinear. For the weak $\ep$-nets application in the projective plane this construction is not applicable yet. The problem is that all axis-parallel lines are incident to only a few, $n,$ points in the infinity. The construction below is described in the real space, $\mathbb{R}^n,$ but it has a natural embedding into the projective space, $\mathbb{PR}^n.$  We will apply two maps to avoid lines being parallel. First we show the construction in $\mathbb{R}^3.$ This ``un-parallelization'' technique was used by Koll\'ar in \cite{Koll}.

Let us define a map, $\nu_3: {\mathbb{R}}^3 \rightarrow {\mathbb{R}}^{7},$ such that 

$$\nu_3: (x_1,x_2, x_3)\longmapsto (x_1, x_2,  x_3, x_1\cdot x_2,  x_1\cdot x_3, x_2\cdot x_3, x_1\cdot x_2\cdot x_3).$$

A parametric equation of a line that is parallel to to the $x_1$ axis, say, is $[t+a,b,c],$ and its image is
$[a, b, c, ba, ca, cd, bca]+t[1,0,0,b, c, 0, bc].$  As the map is injective, non-crossing lines remain non-crossing. Let us check two parallel lines, $\ell_i$ and $\ell_j.$ After the map the two direction vectors are 
$[1,0,0,b_i, c_i, 0, b_ic_i]$ and $[1,0,0,b_j, c_j, 0, b_jc_j].$ If the vectors are parallel then there is a non-zero multiplier $\alpha\in \mathbb{R},$ such that  
$[b_i, c_i, b_ic_i]=\alpha[b_j, c_j, b_jc_j].$ Then $b_ic_j=b_jc_i$ which would contradict to the algebraic independence of the elements of $H_i.$

For the general case let us apply a map, $\nu_n: {\mathbb{R}}^n \rightarrow {\mathbb{R}}^{2^n-1},$ where a point is mapped to all possible products of its coordinates. $$\nu: (x_1,x_2,\ldots, x_n)\longmapsto (x_1,x_2,\ldots, x_n, x_1\cdot x_2,\ldots, x_{n-1}\cdot x_n,\ldots, x_1\cdot x_2\cdot \ldots \cdot x_n).$$

Only the axis parallel collinear triples will stay collinear in the image of $V(\cH)$ in  ${\mathbb{R}}^{2^n-1}$ under the map $\nu_d.$ Now we can project the image back generically to ${\mathbb{R}}^n.$ Other maps, like Veronese embeddings, would be sufficient too.

\medskip

The important parameters of $\cF$ are the following:
      
      $$v(\cF)=k^n,\quad e(\cF)=n\binom{k}{r}k^{n-1},\quad d(\cF)=\frac{rn}{k}\binom{k}{r},\quad \Delta_i=\binom{k-i}{r-i}.$$
    
    In the proof below we chose the parameters
    
    $$k=2^{r^4},\quad n=k^r,\quad  r=(\log n)^{1/5}=(\log n)^{1/4}, \quad  r=t^2, \quad  t=\sqrt r = (\log n)^{1/10}.$$ 
    
    These choices are not optimal. In particular we chose the exponent $1/5$ for the purposes of clarity rather than pushing  the proof to the best possible result. We emphasize that in this section if we say that a set of points of $[k]^n$ is collinear, then we mean that the points are in an axis-parallel line of $[k]^n$.

   We need the following supersaturation lemma.
   
      \begin{lemma}\label{supsatweak}
      Let $\gamma>10r/k$ and $T\subset V(\cF)$ with $|T|>\gamma v(\cF).$ Then
      $$e(\cF[T])\ge (\gamma- r/k)^r e(\cF).$$
         \end{lemma}   
        
        \begin{proof} Fix a coordinate (this can be done in $n$ ways), and a direction, or a representative, of a line given by the other coordinates 
         (this can be done in $k^{n-1}$ ways). On average, a  line contains at least $\gamma k$ points from $T$,
          having $\binom{\gamma k}{r}$ collinear $r$ tuples.  Thus the number of hyperedges in  $T$ is at least  
         $$n\cdot k^{n-1}\cdot \binom{\gamma k}{r}=\frac{\binom{\gamma k}{r}}{\binom{ k}{r}} e(\cF)\ge 
         \left(\frac {\gamma k -r}{k}\right)^r e(\cF).$$
         \end{proof}
              
\medskip
The construction of the set of points is as before:
first  we choose a random $p$-subset of the grid $[k]^n$, then we remove 
 a point from each  collinear $(r+1)$-tuples.
 We shall do an additional sparsening: for every collection of 
     $t$ lines, each containing   $r$ ($p$-random) points and intersecting in one of the non-$p$-random points, we
   remove one  out of the $tr$ points. In other words, after this sparsening there will be no $t$ collinear ($p$-random) $r$-tuples that could be covered by one non-$p$-random point.
 We need to choose $p$ carefully, as we do not want to remove too many points. These restrictions will give an upper bound on $p$.

  Using a similar method to the one we used in the $\ep$-net proof, we shall show that the set we obtain has no large independent subset. From here, we shall obtain a lower bound on $p$. Using  a switching idea of Alon~\cite{Alon}, we will conclude that  no small weak $\ep$-net exists. 
  
  \medskip
  
  The first condition on $p$ comes from the fact that the expected number of collinear $(r+1)$-tuples needs to be  much less than the expected number of points. That is, we need $p^{r+1}n\binom{k}{r+1}k^{n-1}\ll p \cdot k^n$,
  which is satisfied when 
  \begin{equation}\label{condr+1}
  p\ll \frac{r}{k\cdot n^{1/r}} = \frac{r}{k^2}.
  \end{equation}

For the second condition we use a
 key property of the projection (which Alon~\cite{Alon} used as well)
 that  if there are at least  three  collinear $r$-tuples of points passing through the same point, then
this point was a grid point(!). This aids us in the counting of the $t$-tuples of lines, each containing $r$ random grid points, as we can start by their intersection point.
The expected number of $t$-tuples of such lines is, using that each point is in $n$ lines, 
$$k^n\cdot \binom{n}{t}\binom{k}{r}^tp^{rt},$$
where $k^n$ is the number of points, $\binom{n}{t}$ is the number of ways to choose the $t$ directions,
and $\binom{k}{r}^t$ is the number of ways to choose $r$ points from each resulting line, which are kept with probability $p$ each.

  We need this to be much less than $p\cdot k^n$, which is satisfied if

 \begin{equation}\label{condr+2}
  p\ll \frac{1}{k^{1+1/(rt-1)}\cdot n^{t/(rt-1)}}.
  \end{equation}

\noindent For the container lemma we set
$$\tau=\frac{r\cdot 2^r}{k\cdot n^{1/(r-1)}},\quad \ep =10^{-(r+1)}.$$
In order to apply Theorem~\ref{thm-container}, we have to check the condition on $\De(\cF,\tau), \tau, \ep$:

$$\De(\cF,\tau)=2^{\binom{r}{2}-1} \sum_{j=2}^r \frac{\De_j}{d\tau^{j-1} 2^{\binom{j-1}{2}}} \le 
2^{\binom{r}{2}-1} \sum_{j=2}^r \frac{\binom{k-j}{r-j}\cdot k}{r\cdot n\cdot \binom{k}{r}\cdot  \tau^{j-1} 2^{\binom{j-1}{2}}} \le \left(\frac{e r}{n^{1/(r-1)} }\right)^{r-j+1} 2^{-r(j-1)}\le \frac{\ep}{12r!} ,$$
which is satisfied by our choice of parameters. Note that when $j>r/2$ then  $2^{\binom{j}{2}+r(j-1)}$ , and if $j\le r/2$ then $n^{(r-j)/(r-1)}$ is the dominant term in the denominator, which makes the relation true.

The Container Theorem~\ref{thm-container} gives the existence of at most  
$$2^{1000\cdot r\cdot r!^3 v(\cF)\tau\cdot \log(1/\ep)\cdot \log(1/\tau)}\le 2^{v(\cF)r^{4r} \cdot  k^{-1}\cdot n^{-1/(r-1)}}$$
containers,  such that  each independent set is contained in one of them, and each spans at most
$ \ep\cdot  e(\cF)$ hyperedges. By Lemma~\ref{supsatweak}, each container has at most $v(\cF)/9$ vertices.

Now we can give an upper bound on the number of independent sets in the $p$-random subset of $\cH$ of size $m= 2pv(\cH)/3$, by simply multiplying  the number of containers by the number of choices of an $m$-set from a container:

$$
2^{v(\cF)r^{4r} \cdot  k^{-1}\cdot n^{-1/(r-1)}}
\cdot  \binom{  v(\cH)/9}{m}\cdot p^m\le  
2^{v(\cF)r^{4r} \cdot  k^{-1}\cdot n^{-1/(r-1)}}   2^{-m} =o(1),$$
if 
 \begin{equation}\label{cond3}
p\gg \frac{r^{4r}}{n^{1/(r-1)}\cdot k}.
\end{equation}

We need to choose $p$ such that the relations \eqref{condr+1}, \eqref{condr+2} and  \eqref{cond3} are all satisfied:

$$ p:=\frac{r^{4r}\cdot n^{1/2r^2}} {k\cdot n^{1/(r-1)}}.$$

Now we can put together the proof. 
We start with a $p$-random subset of the grid $[k]^n$, and we remove a point from each of the axis-parallel lines containing at least $r+1$ points. Also, from each $t$-tuple of axis-parallel lines intersecting in a (non $p$-random grid) point, where each line was containing $r$ $p$-random points, we remove a point. The choice of $p$ guarantees that with high probability most of the points were not deleted. After this cleaning, we project this point set into the plane in such a way that we do not create new collinear $r$-tuples.
 Call the resulting point set $S$, and write $s:=|S|$. We aim to prove that there is no  weak $\ep^*$-net $T$  of $S$ of size smaller than $t/5s$. The container method gives that every subset of  $S$ of size $3s/4$ contains a collinear $r$-tuple, so $S$ does not contain a subset of size $s/4$ covering all collinear $r$-tuples.  Note that here we increased $2/3$ to $3/4$, in order to have this property with high probability, and also that the actual size of $S$ might differ slightly from its expectation. 
 
 Let $\ep^*=r/s$.  Assume that $T$
   is a weak $\ep^*$-net   of $S$. Using the switching idea of Alon~\cite{Alon}, we   find a slightly larger net (covering set), for whose size we have a lower bound. 
 We partition $T$ into three sets, i.e.,  $T=T_g\cup T_w\cup T_s$, where $T_w=(T-S)\cap V(\cF),\quad T_s=T\cap S$ and $T_g:=T-T_w-T_s$.  
 
 As $T_g$ contains only non-grid points,  we have that each point of $T_g$ covers at most two lines. Hence, $T_g$ can be replaced with a set $T'_g\subset S$ covering the same lines that $T_g$ covered, choosing one point from each line, with $2|T_g|\ge  |T'_g|.$
 Each point of $T_w$ can be replaced 
  with at most $t$ points from $S$ to cover the same set of lines, so we can form 
  $T'_w\subset S$ with  $t|T_w|\ge  |T'_w|.$
    This means that $T'_g\cup T'_w\cup T_s\subset S$ is an  $\ep^*$-net, so it has size at least $s/5$. This implies that $|T|> s/5t$.  

To summarize, we proved the following result about weak nets.
We have a set of size $s=(1+o(1))~p~\cdot~k^n$, where to cover all the lines containing at least $r$ points we need a set of size at least $s/5\sqrt r$. For $\ep^0>0$ very small constant we choose $k$ such that $r=\log^{1/4} k$, and $s=(1+o(1))~p~\cdot~k^n\approx k^{k^r}>1/\ep^0$. For this choice of $k$ and $r$ we run our proof above. We set $\ep^*:=r/s$, and we can get a lower bound on the size of $T$ as a function of $\ep^*$:
$$|T|> \frac{s}{5t}=\frac{1}{\ep^*}\cdot \frac{r}{5t}=\frac{1}{\ep^*\cdot 5t}> \frac{1}{10\ep^*}\log\log^{1/10} 
\frac{1}{\ep^*}.$$

\subsection{Decomposing coverage of the plane}

{\bf Proof of Theorem~\ref{dconst}:} We fix an arbitrary  small constant $\gamma>0$ and $r>r(\gamma)$ (though the method does not require $r$ to be large, the computations are somewhat less tedious under this assumption.) Let $n= 2^{r^4}$. We shall work with the $r$-uniform hypergraph $\cF=\cF (n,n,r)$, i.e., $V(\cF)=[n]^n$, and $E(\cF)$ consists of collinear $r$-tuples of points, where collinearity means that all but one coordinate of the points are the same. We need to compute some parameters of $\cF$:
$$V(\cF)=n^n,\quad e(\cF)=n^n\cdot \binom{n}{r},\quad d=d(\cF)=r\cdot \binom{n}{r},\quad \Delta_i=\binom{n-i}{r-i} \quad \text{when}\quad 2\le i\le r.$$

\noindent Then we set
 $$\ep=0.5 \gamma^{r}, \quad \tau=\frac{r\cdot 2^{r}}{n^{1+1/(r-1)}},\quad  p   =     \frac{ r^{4r}}{n^{1+ 1/(r-1)}},
\quad m={10\gamma pn^n}.$$

Now we apply Theorem~\ref{thm-container} to $\cF$, where first we have to check that $\Delta(\cF,\tau)\le \ep/(12\cdot r!)$; we omit the details of the straightforward but tedious computations. We shall obtain a family of containers with at most 
$$2^{v(\cF)\cdot r^{3r}\cdot n^{-r/(r-1)}}$$ 
sets, each spanning at most $\ep\cdot e(\cF)$ hyperedges, which by Lemma~\ref{supsatweak} means that each set
 has size at most $\gamma\cdot v(\cF)$.

\medskip

Now we sparsen $v(\cF)$, keeping each vertex with probability $p$. Next, we clean the obtained hypergraph, first removing vertices from $\cF$  from lines with at least $r+1$ points. Because the expected number of such lines is much less than the expected number of points (because $p\ll n^{-(r+1)/r}$), we remove only few vertices.
Additionally, 
 we remove vertices from the remaining random point set that are on too many, say
$r^{5r^2-1}= T$, lines. An application of the Chernoff bound gives that there are not many points removed.
We need that in the random hypergraph there are no  independent sets of size $m$. Here we have to prove that the expected number of them, in the container sets, is $o(1)$:
$$2^{v(\cF)\cdot r^{3r}\cdot n^{-r/(r-1)}}\cdot \binom{\gamma\cdot v(\cF)}{m}\cdot  p^m\le
2^{v(\cF)\cdot r^{3r}\cdot n^{-r/(r-1)}}\cdot    
 \left(\frac{\gamma\cdot v(\cF)\cdot p}{10\gamma p\cdot v(\cF)} \right)^{10\gamma pv(\cF)}=o(1).$$
Again, the key observation is that the complement of a point set covering all the lines is an independent set, i.e., small independence number implies that the minimum cover should consist of most of the vertices.

Now we can do the standard projection into the plane, and we obtain the required point system.

\newpage 

\section*{Acknowledgments} 
The authors are thankful to J\'anos Pach who suggested that our method might be applicable to the $\ep$-nets problem. We also thank Noga Alon and Wojciech Samotij for fruitful discussions on the $\ep$-nets problem, and Hong Liu, Adam Zsolt Wagner and the referee for careful reading of the manuscript.

\bibliographystyle{amsplain}


\begin{dajauthors}
\begin{authorinfo}[balogh]
 J\'ozsef Balogh\\
Department of Mathematical Sciences,\\
 University of Illinois at Urbana-Champaign,\\
  Urbana, Illinois 61801, USA\\
jobal\imageat{}math.uiuc.edu.\\
  \url{www.math.uiuc.edu/~jobal}
\end{authorinfo}
\begin{authorinfo}[solymosi]
 J\'ozsef Solymosi\\
 Department of Mathematics, \\
 University of British Columbia \\
 1984 Mathematics Road,\\
Vancouver British Columbia V6T 1Z4, Canada\\
  solymosi\imageat{}math.ubc.ca.\\
 \url{https://www.math.ubc.ca/~solymosi}
 \end{authorinfo}
 \end{dajauthors}

\end{document}